\documentclass[12pt, reqno]{amsart}
\usepackage[utf8]{inputenc}
\usepackage[english]{babel}
\usepackage[T2A]{fontenc}
\usepackage[margin=3cm]{geometry}
\usepackage{amssymb, amsmath, amscd, amsthm}
\usepackage{graphicx}
\usepackage{xcolor}

\title{Topological stability of continuous functions with respect to averaging by measures with locally constant densities}
\author{Sergiy Maksymenko, Oksana Marunkevych}
\address{Institute of Mathematics of NAS of Ukraine, str.\! Tereshchenkivs'ka, 3, 01601, Kyiv, Ukraine}
\email{maks@imath.kiev.ua, oxanamarunkevysh@rambler.ru}


\makeatletter
\newcommand\testshape{family=\f@family; series=\f@series; shape=\f@shape.}
\def\myemphInternal#1{\if n\f@shape%
\begingroup\itshape #1\endgroup\/%
\else\begingroup\bfseries #1\endgroup%
\fi}
\def\myemph{\futurelet\testchar\MaybeOptArgmyemph}
\def\MaybeOptArgmyemph{\ifx[\testchar \let\next\OptArgmyemph
                 \else \let\next\NoOptArgmyemph \fi \next}
\def\OptArgmyemph[#1]#2{\index{#1}\myemphInternal{#2}}
\def\NoOptArgmyemph#1{\myemphInternal{#1}}
\makeatother

\newcommand\RRR{\mathbb{R}}
\newcommand\eps{\varepsilon}

\newtheorem{definition}[subsection]{Definition}
\newtheorem{theorem}[subsection]{Theorem}
\newtheorem{lemma}[subsection]{Lemma}
\newtheorem{remark}[subsection]{Remark}

\newtheorem{example}[subsection]{Example}

\newcommand\Homeo{\mathcal{H}}

\renewcommand\subparagraph[1]{}
\begin{document} 


\begin{abstract}
We present sufficient conditions for topological stability of averagings of piece-wise differentiable functions $f:\RRR\to\RRR$ having finitely many local extremes with respect to measures with locally constant densities.
\end{abstract}

\maketitle 

\section{Introduction}

Let $\mu$ be a probability measure on $[-1,1]$, that is a non-negative $\sigma$-additive measure defined on the Borel algebra of subsets of $[-1,1]$ and such that $\mu[-1,1]=1$.
Then for each continuous function $f\colon (a,b)\to\RRR$ and a number $\alpha>0$ satisfying $2\alpha < b-a$ one can define a new measurable function $f_{\alpha}\colon (a+\alpha, b-\alpha)\to\RRR$ defined by:
\begin{equation}\label{equ:averaging}
f_{\alpha}(x) = \int_{-1}^{1} f(x+t\alpha)  d\mu.
\end{equation}
We will call it an \textit{$\alpha$-averaging} of $f$ with respect to the measure $\mu$.
If $\mu$ has a density function $p$, then $f_{\alpha}$ is a \myemph{convolution} of $f$ with $p$, see Remark~\ref{rem:convolution} below.

Averaging of functions plays an important role in the problem of signals processing and are called \myemph{linear filters}, \cite{Huang:Linfilters1:1981}, \cite{Kenneth:IEEE:1995}, \cite{Milanfar:IEEE:2013}.


The present paper continues the authors work~\cite{MaksymenkoMarunkevych:UMZ:ENG:2015} on topological stability of averagings of continuous functions, see Definitions~\ref{defn:top_stab}, \ref{def:top_stab_rel_averagings}, and~\ref{def:loc_top_stab} below.

\begin{definition}\label{defn:top_stab}
{\rm(e.g.~\cite{Arnold:UMN:ENG:1992}, \cite{Thom:Top:1965}).}
Two continuous functions $f\colon (a,b) \to \RRR$ and $g\colon (c,d) \to \RRR$ are \myemph{topologically equivalent} if there exist orientation preserving homeomorphisms $h\colon  (a,b)\to(c,d)$ and $\phi\colon\RRR\to\RRR$ such that $\phi\circ f = g \circ h$, that is the following diagram is commutative.
\[
\begin{CD}
(a,b) @>{f}>{}>\RRR \\
@V{h}VV @VV{\phi}V \\
(c,d) @>{g}>> \RRR
\end{CD}
\]
\end{definition}
Roughly speaking this means that the graphs of $f_{\alpha}$ and $f$ <<have the same form>>.



\begin{definition}\label{def:top_stab_rel_averagings}
{\rm\cite{MaksymenkoMarunkevych:UMZ:ENG:2015}.}
Let $f\colon \RRR\to\RRR$ be a continuous function and $\mu$ be a probability measure on $[-1,1]$.
Say that $f$ is \myemph{topologically stable} with respect to averaging by measure $\mu$ if there exists $\varepsilon>0$ such that for all $\alpha\in(0,\varepsilon)$ the functions $f$ and $f_{\alpha}$ are topologically equivalent.
\end{definition}


The problem topological stability of averagings has applications to computation of entropy of digital signals, \cite{BandtPompe:PRL:2002}, \cite{AntonioukKellerMaksymenko:DCDSA:2014}, \cite{KellerMaksymenkoStolz:DCDSB:2015}.


Let $C^{0}(\RRR)$ be the space of all continuous functions $\RRR\to\RRR$ and $\Homeo^{+}(\RRR)$ be the group of all orientation preserving homeomorphisms of $\RRR$.
Evidently, $\Homeo^{+}(\RRR)$ consists of all strictly increasing surjective continuous functions $h\colon \RRR\to\RRR$.
Then the product $\Homeo^{+}(\RRR) \times \Homeo^{+}(\RRR)$ acts on the space $C^{0}(\RRR)$ by the following rule: if $(h,\phi)\in \Homeo^{+}(\RRR) \times \Homeo^{+}(\RRR)$ and $f\in C^{0}(\RRR)$, then the result of the action of $(h,\phi)$ on $f$ is the function
\[
\phi \circ f \circ h^{-1}\colon \RRR\to\RRR.
\]
This action is one of the main objects of study in singularities theory, see \cite{GolubitskiGuillemin:StableMats:ENG:1977}, \cite{ArnoldVarchenkoGuseinZade:1}.

Notice that $f,g\in C^{0}(\RRR)$ are topologically equivalent if and only if they belong to the same orbit with respect to the above action of $\Homeo^{+}(\RRR) \times \Homeo^{+}(\RRR)$.

For $f\in C^{0}(\RRR)$ define the following path started at $f$:
\[ 
\gamma_{f}\colon  [0,\infty) \to C^{0}(\RRR),
\qquad 
\gamma_{f}(\alpha) = f_{\alpha}.
\]
Evidently, $f\in C^{0}(\RRR)$ is topologically stable with respect to averaging by measure $\mu$ if and only if a ``small initial part'' of this path, i.e.\! $\gamma_{f}[0,\eps]$ for some $\eps>0$, is contained in the orbit of $f$.

Thus the averaging is a linear operation on $C^{0}(\RRR)$ of all continuous functions $\RRR\to\RRR$, while topological equivalence arises from non-linear actions of the group $\Homeo(\RRR)\times\Homeo(\RRR)$.

\medskip

In~\cite{MaksymenkoMarunkevych:UMZ:ENG:2015} the authors obtained sufficient conditions for topological stability of continuous functions $f:\RRR\to\RRR$ having finitely many local extremes with respect to averagings.
It is shown that this global problem reduces to a stability of germs of $f$ near these local extremes.
In the present paper we will prove that those sufficient conditions are also necessary, see Definition~\ref{def:loc_top_stab} and Theorem~\ref{th:main_theorem_global_2} below.
Thus the problem of global stability is equivalent to a local problem of stability of germs.

\medskip

Let $f\colon (\RRR,a)\to\RRR$ be a germ of continuous function at some $a\in\RRR$, that is $f$ is a continuous function defined on a small interval $(a-\varepsilon, a+\varepsilon)$ for some $\varepsilon$.
If $\alpha<\varepsilon$, then $f_{\alpha}$ is defined on $(a-\varepsilon+\alpha, a+\varepsilon-\alpha)$, and so its germ at $a$, clearly, depends on only of the germ of $f$ at $a$.

Notice that in general the germs $f$ and $f_{\alpha}$ at $a$ are not topologically equivalent: local extremes can move under averaging.
This leads to the following definition:

\begin{definition}\label{def:loc_top_stab}
Say that a germ $f\colon (\RRR,a)\to\RRR$ is \myemph{topologically stable} with respect to averaging by measure $\mu$, if there exists $\eps>0$ such that for each $\alpha\in(0,\varepsilon)$ there exist numbers $c_1,c_2, d_1,d_2\in(a-\eps, a+\eps)$ satisfying $c_1 < a < c_2$, $d_1 < d_2$ and such that the restrictions
\begin{align*}
f|_{(c_1,c_2)}\colon  & \ (c_1,c_2) \ \to \ \RRR, &
f_{\alpha}|_{(d_1,d_2)}\colon  & \ (d_1,d_2) \ \to \ \RRR
\end{align*}
are topologically equivalent.
\end{definition}

\begin{theorem}\label{th:main_theorem_global_2}{\rm c.f.~\cite{MaksymenkoMarunkevych:UMZ:ENG:2015}}
Let $\mu$ be a probability measure on $[-1,1]$ and $f\colon \RRR \to \RRR$ be a continuous function having only finitely many local extremes $x_1,\ldots,x_n$.
Suppose that the values $f(x_i)$, $(i=1,\ldots,n)$, are mutually distinct and differ from $\lim\limits_{x\to-\infty} f(x)$ and $\lim\limits_{x\to+\infty} f(x)$.
Then the following conditions are equivalent:
\begin{enumerate}
\item[\rm(a)]
$f$ is topologically stable with respect to averagings by measure $\mu$;
\item[\rm(b)]
for each $i=1,\ldots,n$ the germ $f\colon (\RRR,x_i) \to \RRR$ at $x_i$ is topologically stable with respect to averaging by measure $\mu$.
\end{enumerate}
\end{theorem}
The implication (b)$\Rightarrow$(a) is established in~\cite{MaksymenkoMarunkevych:UMZ:ENG:2015}.
We will prove that (a)$\Rightarrow$(b).
Thus for functions ``of general position'' the problem completely reduces to study local stability of germs at local extremes.


In~\cite{MaksymenkoMarunkevych:UMZ:ENG:2015} the authors also obtained sufficient conditions for topological stability of germs with respect to averagings by discrete measures with finite supports.
In the present paper we give sufficient conditions for topological stability of germs with respect to measures with piece wise continuous  (and in particular with piece wise constant) densities, see Theorems~\ref{th:perturb_averaging_stab} and~\ref{th:LR_averaging_stab}.


\begin{remark}\label{rem:convolution}\rm
Suppose that $\mu$ has a density $p\colon [-1,1]\to\RRR$, that is a measurable function such that $\mu(A) = \int_{A} p(t) dt$ for each Borel subset $A\subset[-1,1]$.
For $\alpha>0$ define the function $p_{\alpha}\colon [-\alpha,\alpha]\to\RRR$ and measure $\mu_{\alpha}$ on $[-\alpha,\alpha]$ by the formulas:
\begin{align*}
p_{\alpha}(s) &= \dfrac{p(s/\alpha)}{\alpha}, &
\mu(A) &= \int_{A} p_{\alpha}(s) ds.
\end{align*}
Then 
\begin{align*}
\mu_{\alpha}[-1,1] &= \int_{-\alpha}^{\alpha} p_{\alpha}(s) ds =
\int_{-\alpha}^{\alpha} \dfrac{p(s/\alpha)}{\alpha} ds  \\
&= \int_{-\alpha}^{\alpha} p(s/\alpha) d(s/\alpha) =
\int_{-1}^{1} p(t) dt =1,
\end{align*}
and so $\mu_{\alpha}$ is also a probability measure.
Moreover,
\begin{align*}
f_{\alpha}(x) &= \int\limits_{-1}^{1} f(x+t\alpha) p(t) dt = 
\int\limits_{-\alpha}^{\alpha} f(x+s) p(s/\alpha) d(s/\alpha) \\ 
&= \int\limits_{-\alpha}^{\alpha} f(x+s) p_{\alpha}(s) ds.
\end{align*}
The last integral is called a \myemph{convolution} of $f$ and $p_{\alpha}$ and denoted by $f* p_{\alpha}$.

Usually in the formula for convolution one uses $f(x-s)$ instead of $f(x+s)$.
But this is not essential and plays a role only for certain useful algebraic properties of convolution.
For our purposes it will be convenient to use sign <<+>>.
\end{remark}

\section{Proof of Theorem~\ref{th:main_theorem_global_2}}
The implication (b)$\Rightarrow$(a) is established in~\cite{MaksymenkoMarunkevych:UMZ:ENG:2015}. 
Let us show that (a)$\Rightarrow$(b).

Suppose that $f$ is topologically stable with respect to averagings by $\mu$.
This means that there exists $\eps>0$ such that for each $\alpha\in(0,\eps)$ there are two homeomorphisms $h_{\alpha},\phi_{\alpha} \in \Homeo^{+}(\RRR)$ satisfying $\phi_{\alpha}\circ f_{\alpha}  = f \circ h_{\alpha}$.
In particular, $f_{\alpha}$ has $n$ local extremes $h_{\alpha}(x_i)$, $i=1,\ldots,n$, and takes at them values $f_{\alpha}(h_{\alpha}(x_i)) =\phi_{\alpha}(f(x_i))$.
We should prove that the germ of $f$ at $x_i$ is topologically stable with respect to averaging by measure $\mu$.

Decreasing $\eps$ one may assume that
\begin{equation}\label{equ:dist_xi1_xi}
x_{i+1} - x_{i} > 4\eps
\end{equation}
for all $i=1,\ldots,n-1$.
Let $\alpha \in(0,\eps)$.
Since $f$ is strictly monotone on the intervals 
\[
(-\infty, x_1), 
\ (x_1,x_2),
\ \cdots, \
(x_n, +\infty),
\]
one easily checks, see~\cite[Lemma~2.1]{MaksymenkoMarunkevych:UMZ:ENG:2015}, that $f_{\alpha}$ is strictly monotone on
\[
(-\infty, x_1-\alpha), 
\ (x_1+\alpha,x_2-\alpha),
\ \cdots, \
(x_n+\alpha, +\infty).
\]
Therefore, $h_{\alpha}(x_i) \in [x_i-\alpha,x_i+\alpha]$.
Moreover, it follows from~\eqref{equ:dist_xi1_xi} that $h_{\alpha}(x_i)$ is a unique local extreme of $f_{\alpha}$ on the interval  $(x_i-2\alpha,x_i+2\alpha)$.
Put
\[
(c_1,c_2) \ = \ (x_i-\alpha,x_i+\alpha) \ \cap \ h_{\alpha}^{-1}(x_i-2\alpha,x_i+2\alpha).
\]
\[
(d_1,d_2) \ = \ h_{\alpha}(c_1,c_2).
\]
Then the restrictions $f|_{(c_1,c_2)}$ and $f_{\alpha}|_{(d_1,d_2)}$ are topologically equivalent, that is $\phi_{\alpha}\circ f_{\alpha}  = f \circ h_{\alpha}$.

\section{Piece wise differentiable functions}
In this section we will give sufficient conditions for topological stability of local extremes with respect to averaging by measures having locally continuous densities, see Theorem~\ref{th:perturb_averaging_stab}.

\begin{definition}\label{def:piece_wise_diff_func}
A function $f\colon [a,b]\to\RRR$ is called \myemph{piece wise continuous}, or \myemph{piece wise $0$-differentiable}, if $f$ is continuous everywhere on $[a,b]$ except for a finitely many points $t_1, \ldots, t_n \in (a,b)$, and at each $t_i$ there exist finite left and right limits $\lim\limits_{t\to t_i-0} f(t)$ and $\lim\limits_{t\to t_i+0} f(t)$.
In this case we will also write that $f\in C^0([a,b],t_1,\ldots,t_n)$.

Say that a continuous function $f\colon [a,b]\to\RRR$ is \myemph{piece wise $k$-differentiable}, $k\geq1$, if there exist finitely many points $t_1,\cdots,t_n\in(a,b)$ such that $f$ has continuous derivatives of all orders $\leq k$ on $[a,b]\setminus\{t_1,\ldots,t_n\}$ and for each $i=1,\ldots,n$ and $s=1,\ldots,k$ there exist finite left and right limits
\begin{align*}
f_l(t) &= \lim\limits_{t\to t_i-0} f^{(s)}(t), &
f_r(t) &= \lim\limits_{t\to t_i+0} f^{(s)}(t).
\end{align*}
In this case we will also write $f \in C^{k}([a,b],t_1,\ldots,t_n)$.
\end{definition}
Evidently, the sum and the product of piece wise continuous ($k$-differentiable) functions is piece wise continuous ($k$-differentiable) as well.
Moreover, for $k\geq1$ the derivative of a piece wise $(k+1)$-differentiable function can be defined (in arbitrary way) at discontinuity points to give a piece wise $k$-differentiable function.

The following lemma is well known for continuously differentiable functions:
\begin{lemma}\label{lm:f_is_convex}
Let $f\colon [a,b]\to\RRR$ be a continuous function.
Suppose that one of the following conditions holds true:
\medskip
\begin{enumerate}
\item $f\in C^1([a,b], t_1,\ldots,t_n)$ and $f'(x)<f'(y)$ for all $x<y \in [a,b] \setminus \{t_1,\ldots,t_n\}$;
\medskip
\item $f\in C^2([a,b], t_1,\ldots,t_n)$, $f''(x)>0$ for all $x \in [a,b] \setminus \{t_1,\ldots,t_n\}$, and $\lim\limits_{t\to t_i-0} f'(t) \leq \lim\limits_{t\to t_i+0} f'(t)$ for all $i=1,\ldots,n$.
\end{enumerate}
\medskip
Then $f$ is strictly convex.
\end{lemma}
\begin{proof}
We will use the following notation for the left and right derivatives of $f'$:
\begin{align*}
f'_l(x) &= \lim\limits_{t\to x-0} f'(t), &
f'_r(x) &= \lim\limits_{t\to x+0} f'(t).
\end{align*}

(2)$\Rightarrow$(1).
The assumption $f''(x)>0$ for all $x \in [a,b] \setminus \{t_1,\ldots,t_n\}$ means that $f'$ strictly increases on each of the segments 
\[ [a,t_1], \ [t_1,t_2], \ \ldots, \ [t_{n-1},t_n], \ [t_n, b]. \]
Moreover, we also have that $f'_l(t_i)\leq f'_r(t_i)$ for all $i=1,\ldots,n$.
Therefore $f'(x)<f'(y)$ for all $x<y \in [a,b] \setminus \{t_1,\ldots,t_n\}$, that is condition (1) holds true.

\medskip

(1) Since $f'$ is piece wise continuous and strictly increases on $[a,b] \setminus \{t_1,\ldots,t_n\}$, it follows that $f'_l(t) \leq f'_r(t)$ for all $t\in(a,b)$ and that both functions $f'_l$ and $f'_r$ strictly increase.

Let $x<y\in[a,b]$ and $t\in(0,1)$. 
Then
\[
f(x) + (y-x) f'_r(x) \ < \  f(y) = f(x) + \int_{x}^{y} f'(t)dt \ < \ f(x) + (y-x) f'_l(y). 
\]
In particular, if $s\in(0,1)$ and $z=(1-s)x + sy \in(x,y)$, then
\begin{align*}
f(z) &< f(x) + (z-x)f'_l(z) \ = \ f(x) + s(y-x)f'_l(z), \\
f(z) &< f(y) - (y-z)f'_r(z) \ = \ f(y) - (1-s)(y-x)f'_r(z).
\end{align*}
Multiplying the first inequality by $1-s$, the second inequality by $s$, adding them and taking to account that $f'_l(z)-f'_r(z)\leq 0$, we get that
\begin{align*}
f(z) &< (1-s)f(x) + sf(y) + s(1-s)(y-x)\bigl(f'_l(z)-f'_r(z)\bigr) \\ &\leq (1-s)f(x) + sf(y).
\end{align*}
This proves strict convexity of $f$.
\end{proof}

In what follows we will assume that $p\colon [-1,1]\to[0,+\infty)$ is a piece wise continuous function such that $\int_{-1}^{1}p(t) dt = 1$ and $\mu$ is the corresponding probability measure on the Borel algebra $\mathcal{B}[-1,1]$ defined by the formula:
\begin{equation}\label{equ:loc_const_measure}
\mu(A) = \int_{A} p(t) dt, \qquad A \in \mathcal{B}[-1,1].
\end{equation}

\begin{lemma}\label{lm:fa_kdiff}
Let $f\colon [a,b]\to\RRR$ be a continuous function and \[ f_{\alpha}\colon [a+\alpha, b-\alpha]\to\RRR \] be its averaging by measure $\mu$.
Then $f_{\alpha}$ belongs to the class $C^1$.

If $f$ is also piece wise $k$-differentiable (resp. belongs to the class $C^k$) for $k\geq1$, then $f_{\alpha}$ is piece wise $(k+1)$-differentiable (resp. belongs to the class $C^{k+1}$)
\end{lemma}
\begin{proof}
Notice that
\[
f_{\alpha}(x) = \int_{-1}^{1} f(x+t\alpha)p(t)dt = \sum_{i=0}^{n} \int_{t_i}^{t_{i+1}} f(x+t\alpha)p(t)dt.
\]
Since $f$ is continuous, its derivative $f'$ is given by the following formula:
\begin{equation}\label{equ:fprime_alpha_formula}
f'_{\alpha}(x) = \sum_{i=0}^{n} \Bigl( f_l(x+t_{i+1}\alpha) p_l(t_{i+1}) - f_r(x+t_{i}\alpha) p_r(t_{i}) \Bigr),
\end{equation}
and so it is continuous as well.
It also follows that $f'_{\alpha}$ is piece wise $k$-differentiable as well as $f$.
Therefore $f_{\alpha}$ is piece wise $(k+1)$-differentiable.
Moreover,
\begin{equation}\label{equ:s-derivative}
f^{(s)}_{\alpha}(x) = \sum_{i=0}^{n} \Bigl( f_l^{(s-1)}(x+t_{i+1}\alpha) p_l(t_{i+1}) - f_r^{(s-1)}(x+t_{i}\alpha) p_r(t_{i}) \Bigr)
\end{equation}
for all $x$ at which the right hand side is continuous.

If $f$ belongs to the class $C^k$, then, in particular, $f = f_l = f_r$, whence we get from~\eqref{equ:fprime_alpha_formula} that $f_{\alpha}$ belongs to the class $C^{k+1}$.
\end{proof}

\begin{lemma}\label{lm:averaging_stab}
Let $f\colon [-\eps,\eps]\to\RRR$ be a continuous function satisfying the following conditions:
\begin{itemize}
\item[(a)] $f$ strictly increases on $[-\eps,0]$ and strictly decreases of $[0, +\eps]$;
\item[(b)] $f'_{\alpha}$ strictly increases.
\end{itemize}
Then the germ of $f$ at $0$ is topologically stable with respect to the measure $\mu$.
\end{lemma}
\begin{proof}
Since $f$ is continuous, we get from Lemma~\ref{lm:fa_kdiff} that the averaging $f_{\alpha}$ is a continuously differentiable function.
By assumption (b) $f'_{\alpha}$ strictly increases, whence by (1) of Lemma~\ref{lm:f_is_convex} we obtain that $f_{\alpha}$ is strictly convex function.
Since $f_{\alpha}$ decreases on a neighborhood of the point $-\eps+\alpha$ and increases in a neighborhood of $\eps-\alpha$, if follows that $f_{\alpha}$ has a unique local minimum $x_{\alpha}$ on the segment $[-\eps+\alpha,\eps-\alpha]$.
This implies that the germ $f$ at $0$ is topologically equivalent to the germ of $f_{\alpha}$ at $x_{\alpha}$.
\end{proof}

\begin{theorem}\label{th:perturb_averaging_stab}
Let $f,g\colon [-\eps,\eps]\to\RRR$ be two piece wise $1$-differentiable functions and $h=f-g$. 
Suppose the the following conditions hold:
\begin{itemize}
\item[(a)] $f$ and $g$ strictly decrease on $[-\eps,0]$ and strictly increase on $[0, +\eps]$;
\item[(b)] there exists $C>0$ such that for all $x\in[-\alpha,\alpha]$ the following inequality holds:
\[f''_{\alpha}(x) \geq C \alpha\,;\]
\item[(c)] the derivative $h' = g' - f'$ is continuous at $0$ and $h'(0)=0$.
\end{itemize}
Then the germ of $g$ at $0$ is topologically stable with respect to averagings by measure $\mu$.
\end{theorem}
\begin{proof}
Notice that condition (b) guarantees that $f'_{\alpha}$ strictly increases, whence we obtain from (a) and Lemma~\ref{lm:averaging_stab} that $f$ is topologically stable with respect to averagings by measure $\mu$.
We should prove that under condition (c) the function $g=f+h$ (<<perturbation>> of $f$ with $h$) is also topologically stable with respect to averagings by $\mu$.

Since $g$ is continuous and piece wise $1$-differentiable, we get from Lemma~\ref{lm:fa_kdiff} that $g'_{\alpha}$ continuous, and $g''_{\alpha}$ is piece wise continuous.
Moreover, it follows from (a) that for $\alpha < \eps$ the function $g_{\alpha}$ strictly decreases on $[-\eps+\alpha, -\alpha]$ and strictly increases on $[\alpha, \eps-\alpha]$.
In particular,
\begin{align*}
& g'_{\alpha}(-\alpha)<0,
&
&g'_{\alpha}(\alpha)>0.
\end{align*}
Therefore it suffices to show that $\varliminf\limits_{y\to x} g''_{\alpha}(x)>0$ for $x\in [-\alpha,\alpha]$ and all small $\alpha>0$.
This will imply that $g'_{\alpha}$ strictly increases on $[-\alpha,\alpha]$, whence $g_{\alpha}$ will have a unique local minimum.

Since $h'$ is continuous at $0$ and $h(0)=0$, we have that $h(x) = x k(x)$, where
\[
k(x) = \int_0^1 h'(tx) dt.
\]
In particular, $k$ is continuous and $k(0)= h'(0)=0$.
Let \[ P = \sup\limits_{t\in[-1,1]} p(t)\] and $n$ be the number of discontinuity points of the density $p$ of measure $\mu$, see Eq.~\eqref{equ:loc_const_measure}.
Then there exists $\delta>0$ such that $|k(x)|<\frac{C}{8Pn}$ for all $x\in[-\delta,\delta]$.

Let $\alpha < \delta/2$. 
Then for all $x\in[-\alpha,\alpha]$ and $i=0,\ldots,n+1$ we have the following inequality:
\[
|x-t_i\alpha| < |x| + |t_i| \alpha \leq \alpha + \alpha = 2\alpha < \delta,
\]
whence
\[
|h'(x-t_i\alpha)| = |x-t_i\alpha| \cdot |k(x-t_i\alpha)| \leq 2\alpha \cdot \frac{C}{8Pn} =  \frac{C\alpha}{4Pn}.
\]
Now we get from Lemma~\ref{lm:fa_kdiff} that at each $x\in[-\alpha,\alpha]$, where $h''_{\alpha}$ is continuous, we have the inequality:
\begin{align*}
|h''_{\alpha}(x)| &\leq \sum_{i=0}^{n} \Bigl| h_r'(x-t_{i+1}\alpha) p_r(t_{i+1}) - h_l'(x-t_{i}\alpha) p_l(t_{i}) \Bigr| \leq   \frac{C\alpha}{4Pn} \cdot 2Pn = \frac{C\alpha}{2}.
\end{align*}
Therefore
\[
\varlimsup\limits_{y \to x} |h''_{\alpha}(y)| \leq \frac{C\alpha}{2},
\]
and so
\begin{align*}
\varliminf\limits_{y \to x} g''_{\alpha}(y) &= 
\varliminf\limits_{y \to x} \bigl( f''_{\alpha}(x) + h''_{\alpha}(x) \bigr) 
\geq \varliminf\limits_{y \to x} f''_{\alpha}(x) - 
\varlimsup\limits_{y \to x} |h''_{\alpha}(y)| \geq C\alpha - \frac{C\alpha}{2} = \frac{C\alpha}{2} > 0.
\end{align*}
Thus $g'_{\alpha}$ strictly increases.
Theorem~\ref{th:perturb_averaging_stab} is completed.
\end{proof}

\section{Piece wise constant densities}
Let 
\[-1 = t_0 < t_{1} < \cdots < t_{n} < t_{n+1} = 1\]
be an increasing sequence of numbers, $p_{0},\ldots,p_{n} \in [0,+\infty)$ be some non-negative numbers such that $p_{i}\not=p_{i+1}$ for $i=0,\ldots,n-1$.
Define a piece wise constant function $p\colon [-1,1] \to [0,+\infty)$ by the formula:
\[ p[t_{i},t_{i+1}) = p_i, \qquad i=0,\ldots,n-1\]
\[ p[t_{n},t_{n+1}]=p_n,\]
see Figure~\ref{fig:loc_const_density}.
\begin{figure}[h]
\includegraphics[height=2cm]{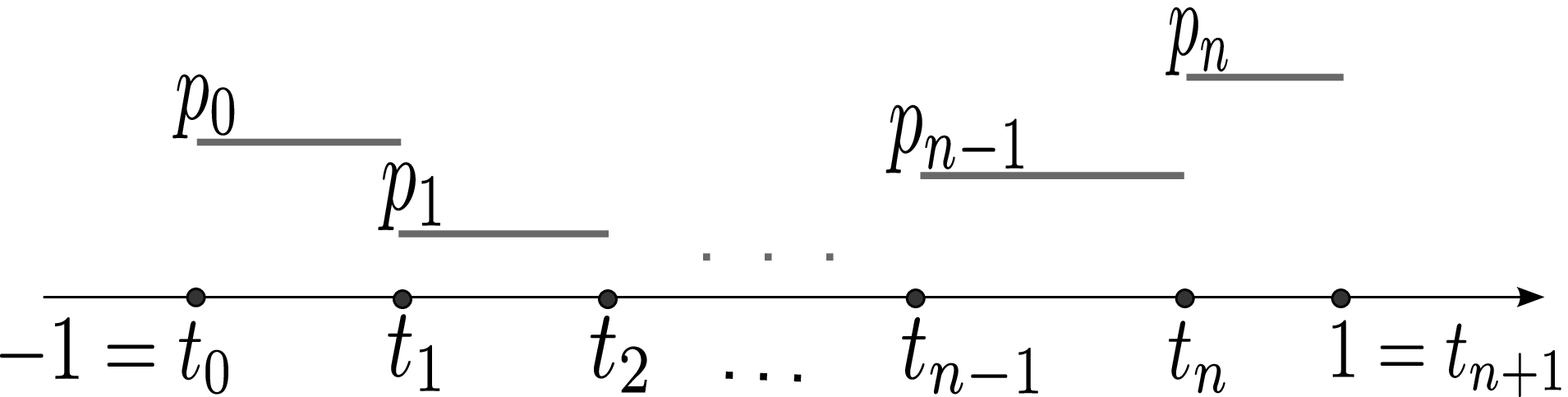}
\protect\caption{}\label{fig:loc_const_density}
\end{figure}

Assume also that
\begin{equation}\label{equ:measure_mu}
\int_{-1}^{1} p(t) dt = \sum_{i=0}^{n} (t_{i+1}-t_{i})p_i = 1.
\end{equation}
Then $p$ is the density function for the probability measure $\mu$ on the Borel algebra of $[-1,1]$ defined by
\[
\mu(A) = \int_{A} p(t)dt, \qquad A \in \mathcal{B}[-1,1].
\]
Hence for each continuous function $f\colon \RRR\to\RRR$ its $\alpha$-averaging $f_{\alpha}\colon \RRR\to\RRR$ by $\mu$ is given by:
\begin{equation}\label{equ:averaging_loc_const_p}
f_{\alpha}(x) = \int_{-1}^{1} f(x+\alpha t)  d\mu = \int_{-1}^{1} f(x+\alpha t) p(t) dt = \sum_{i=0}^{n} p_i \int_{t_{i}}^{t_{i+1}} f(x+\alpha t) dt.
\end{equation}
Notice that then
\begin{equation}\label{equ:averaging_derivative}
f'_{\alpha}(x) = \sum_{i=0}^{n} p_i \int_{t_{i}}^{t_{i+1}} f'(x+\alpha t) dt =
\sum_{i=0}^{n} \bigl( f(x+\alpha t_{i+1}) - f(x+\alpha t_{i}) \bigr) p_i,
\end{equation}
which is a particular case of Eq.~\eqref{equ:fprime_alpha_formula}.

\begin{theorem}\label{th:LR_averaging_stab}
Let $g\colon [-\eps,\eps]\to\RRR$ be a piece wise $1$-differentiable function, satisfying the following conditions:
\begin{itemize}
\item[(a)] $g$ strictly decreases on $[-\eps,0]$ and strictly increases on $[0, +\eps]$;
\item[(b)] there exist finite limits
\begin{align*}
L &= \lim_{x\to 0-0} g'_{x}, & 
R &= \lim_{x\to 0+0} g'_{x}.
\end{align*}
\end{itemize}
For $i=0,\ldots,n+1$ define the following numbers
\begin{align*}
X_i &:= L\mu[t_0,t_i] + R\mu[t_i,t_{n+1}]  = L \sum_{j=0}^{i-1} (t_{j+1}-t_{j})p_j + 
R \sum_{j=i-1}^{n} (t_{j+1}-t_{j})p_j,
\end{align*}
so, in particular,
\[
L \ = \ X_{n+1} \ \leq \ X_{n} \ \leq \ \cdots \ \leq \ X_1 \ \leq \ X_{0} \ = \ R.
\]
Suppose that for each $i\in\{0,\ldots,n\}$ at least on of the numbers $X_{i}$ and $X_{i+1}$ is non-zero.
Then the germ of $g$ at $0$ is topologically stable with respect to the averagings by measure $\mu$.
\end{theorem}

The proof of Theorem~\ref{th:LR_averaging_stab} is based on the following lemma:
\begin{lemma}\label{lm:Lx_Rx}
Let $f\colon \RRR\to\RRR$ be a continuous function defined by
\[
f(x) = 
\begin{cases}
Lx, & x\leq 0 \\
Rx, & x>0.
\end{cases}
\]
Then
\begin{align}\label{equ:derivative_LxRx}
\frac{1}{\alpha} \cdot f'_{\alpha}(x) &= 
\begin{cases}
X_{n+1} = L, & x< -\alpha t_{n+1} = -\alpha, \\
X_{i+1} + \dfrac{x+\alpha t_{i+1}}{t_{i+1}-t_i} (X_{i}-X_{i+1}), &  -\alpha t_{i+1} < x < -\alpha t_{i},  \ 1\leq i \leq n \\
X_{0} = R, & -t_{0}\alpha = \alpha < x, \\
\end{cases}
\medskip
\\
\medskip
\label{equ:second_derivative_LxRx}
f''_{\alpha}(x) &= 
\begin{cases}
0, & x< -\alpha t_{n+1} = -\alpha, \\
\dfrac{X_{i}-X_{i+1}}{t_{i+1}-t_i} \ \alpha,  &  -\alpha t_{i+1} < x < -\alpha t_{i} \\ 
0, & -t_{0}\alpha = \alpha < x.
\end{cases}
\end{align}
\end{lemma}

\medskip

\begin{proof}[Proof of Theorem~\ref{th:LR_averaging_stab}.]
It suffices to check conditions (a)-(c) of Theorem~\ref{th:perturb_averaging_stab} for $f$ from Lemma~\ref{lm:Lx_Rx} and $g$.
Condition (a) holds trivially.

Let 
\[  C = \min\limits_{i=0,\ldots,n} \dfrac{X_{i}-X_{i+1}}{t_{i+1}-t_i}.  \]
Then it follows from Eq.~\eqref{equ:second_derivative_LxRx} and assumption that either $X_{i+1}$ or $X_{i}$ is non-zero for all $i\in\{0,\ldots,n\}$, that $C>0$ and $f''_{\alpha}(x) > C \alpha$ for all $x\in[-\alpha,\alpha]$.
This implies condition (b).

Finally put $h = f - g$.
Then
\begin{align*}
\lim\limits_{x\to 0-0} h'(x) &= \lim\limits_{x\to 0-0} f'(x)  -  \lim\limits_{x\to 0-0} g'(x)  = L-L = 0 \\
\lim\limits_{x\to 0+0} h'(x) &= \lim\limits_{x\to 0+0} f'(x)  -  \lim\limits_{x\to 0+0} g'(x)  = R-R = 0,
\end{align*}
whence $h'$ is continuous at $0$ and $h'(0)=0$.
Thus condition (c) is also satisfied.
Therefore by Theorem~\ref{th:perturb_averaging_stab} the germ of $g$ at $0$ is topologically stable with respect to averagings by $\mu$.
\end{proof}

\begin{proof}[Proof of Lemma~\ref{lm:Lx_Rx}.]
For $n-1\geq i \geq 0$ denote
\[
\Delta_i(x) = \bigl( f(x+\alpha t_{i+1}) - f(x+\alpha t_{i}) \bigr)\,p_i.
\]
Then by~\eqref{equ:averaging_derivative},
\[
f'_{\alpha}(x) = \sum_{i=0}^{n}\Delta_i(x).
\]
Consider three cases.

\medskip

a) If $x + \alpha t_{i} < x + \alpha t_{i+1} < 0$ for some $i=0,\ldots,n$, then
\begin{align*}
\Delta_i(x) &= \bigl( L(x+\alpha t_{i+1}) - L(x+\alpha t_{i}) \bigr)\,p_i =
\alpha L (t_{i+1} - t_{i})\,p_i  = \alpha L \mu[t_{i}, t_{i+1}].
\end{align*}

\medskip

b) Suppose $x + \alpha t_{i} \leq 0 \leq x + \alpha t_{i+1}$ for some $i=0,\ldots,n-1$.
This condition is equivalent to the assumption that $x\in[-\alpha t_{i+1}, -\alpha t_{i}]$.
\begin{figure}[h]
\includegraphics[height=2cm]{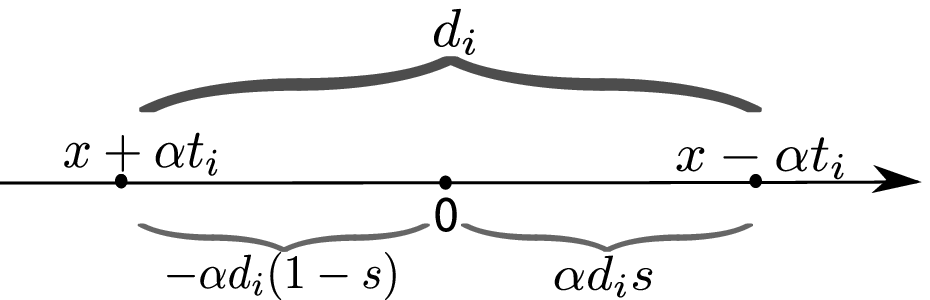}
\protect\caption{}\label{fig:xati_xati1}
\end{figure}

Put
\begin{align*}
d_i &= t_{i+1} - t_i, &
s &= \frac{x+\alpha t_{i+1}}{\alpha d_i},
\end{align*}
see Figure~\ref{fig:xati_xati1}.
Then
\begin{align*}
1-s &= -\frac{x+\alpha t_{i}}{\alpha d_i}, &
x &= -\alpha t_{i+1}(1-s) -\alpha t_i s,
\end{align*}
and so
\begin{align*}
\Delta_i(x) &= \bigl( R(x+\alpha t_{i+1}) - L(x+\alpha t_{i}) \bigr)\, p_i = \bigl( (1-s) L + s R \bigr)\, \alpha \, p_i d_i.
\end{align*}

\medskip

c) If $0 < x + \alpha t_{i} < x + \alpha t_{i+1}$ for some $i=n-1,\ldots,0$, then similarly to the case a) we get that
\begin{align*}
\Delta_i(x) &= \bigl( R(x+\alpha t_{i+1}) - R(x+\alpha t_{i}) \bigr)\, p_i =
\alpha R (t_{i} - t_{i+1}) p_i = \alpha R \mu[t_{i+1}, t_{i}].
\end{align*}

\medskip

Now we can prove Eq.~\eqref{equ:derivative_LxRx} for $f'_{\alpha}$.
Suppose that $x\leq \alpha = \alpha t_0$.
Then $x+\alpha t_{i} < x+\alpha t_{n+1} \leq 0$ for all $i$, whence 
\begin{align*}
f'_{\alpha}(x) &= \sum_{j=0}^{n}\Delta_j(x) = 
\sum_{j=0}^{n} \alpha L \mu[t_{j}, t_{j+1}] =
\alpha L  \sum_{j=0}^{n} \mu[t_{j}, t_{j+1}] = \alpha L \mu[-1,1]=\alpha L.
\end{align*}

If in the case b), $x = -\alpha t_{i+1}(1-s) -\alpha t_i s \in[-\alpha t_{i+1}, -\alpha t_{i}]$ for some $i=0,\ldots,n$, then
\begin{align*}
\frac{1}{\alpha} \cdot f'_{\alpha}(x) &=
\sum_{j=0}^{i-1} L \mu[t_{j}, t_{j+1}] +
\bigl( (1-s)L + sR\bigr)\, \alpha \,\mu[t_i,t_{i+1}] + \sum_{j=i+1}^{n} R \mu[t_{j}, t_{j+1}]  \\[1ex]
&= L \mu[t_{0}, t_{i}] + \bigl( (1-s)L + sR\bigr)\, \alpha \,\mu[t_i,t_{i+1}] + R \mu[t_{i+1}, t_{n+1}] \\[1ex]
&= (1-s) \bigl( L \mu[t_{0}, t_{i+1}] + R \mu[t_{i+1}, t_{n+1}] \bigr)+ s \bigl( L \mu[t_{0}, t_{i}] + R \mu[t_{i}, t_{n+1}] \bigr) \\[1ex]  &= (1-s) X_{i+1} + s X_{i}  =  X_{i+1} + s (X_{i}-X_{i+1}) \\[1ex]
&= X_{i+1} + \frac{x+\alpha t_{i+1}}{t_{i+1}-t_i} (X_{i}-X_{i+1}).
\end{align*}

Finally, if $\alpha = t_{n+1}\alpha \leq x$, то $0 \leq x+\alpha t_{0} <  x+\alpha t_{i}$ for all $i$, and so
\begin{align*}
f'_{\alpha}(x) &= \sum_{j=0}^{n}\Delta_j(x) = 
\sum_{j=0}^{n} \alpha R \mu[t_{j}, t_{j+1}] \\ &=
\alpha R\sum_{j=0}^{n} \mu[t_{j}, t_{j+1}] = \alpha R \mu[-1,1]=\alpha R.
\end{align*}
Lemma is completed.
\end{proof}

\begin{example}\rm
Let us show that if in Theorem~\ref{th:main_theorem_global_2} $X_{i+1}=X_{i}=0$ for some $i$, then the function $g$ can be not tologically stable with respect to $\mu$.
Define the function $f\colon \RRR\to\RRR$ and the density $p\colon [-1,1]\to\RRR$ as follows:
\begin{align*}
f(x) &=
\begin{cases}
-x, & x\leq 0, \\
2x, & x\geq 0
\end{cases}
&
p(x) &= 
\begin{cases}
1, & x\in[-1,-0.5], \\
0,  & x\in(-0.5,0], \\
0.25,  & x\in(0,1),
\end{cases}
\end{align*}
see Figure.~\ref{fig:counterexample_non_stab}.
\begin{figure}[h]
\begin{tabular}{ccc}
\includegraphics[width=0.4\textwidth]{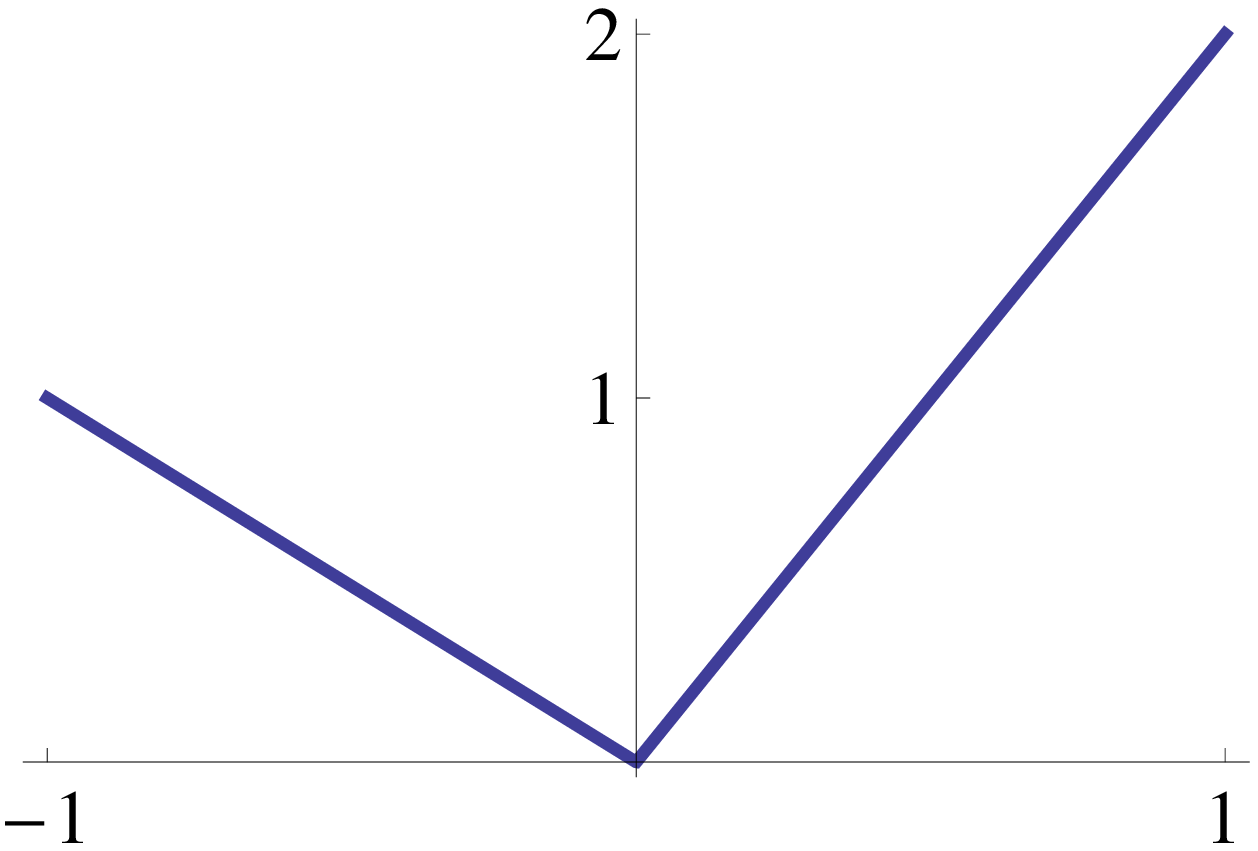} & \qquad\qquad & \includegraphics[width=0.4\textwidth]{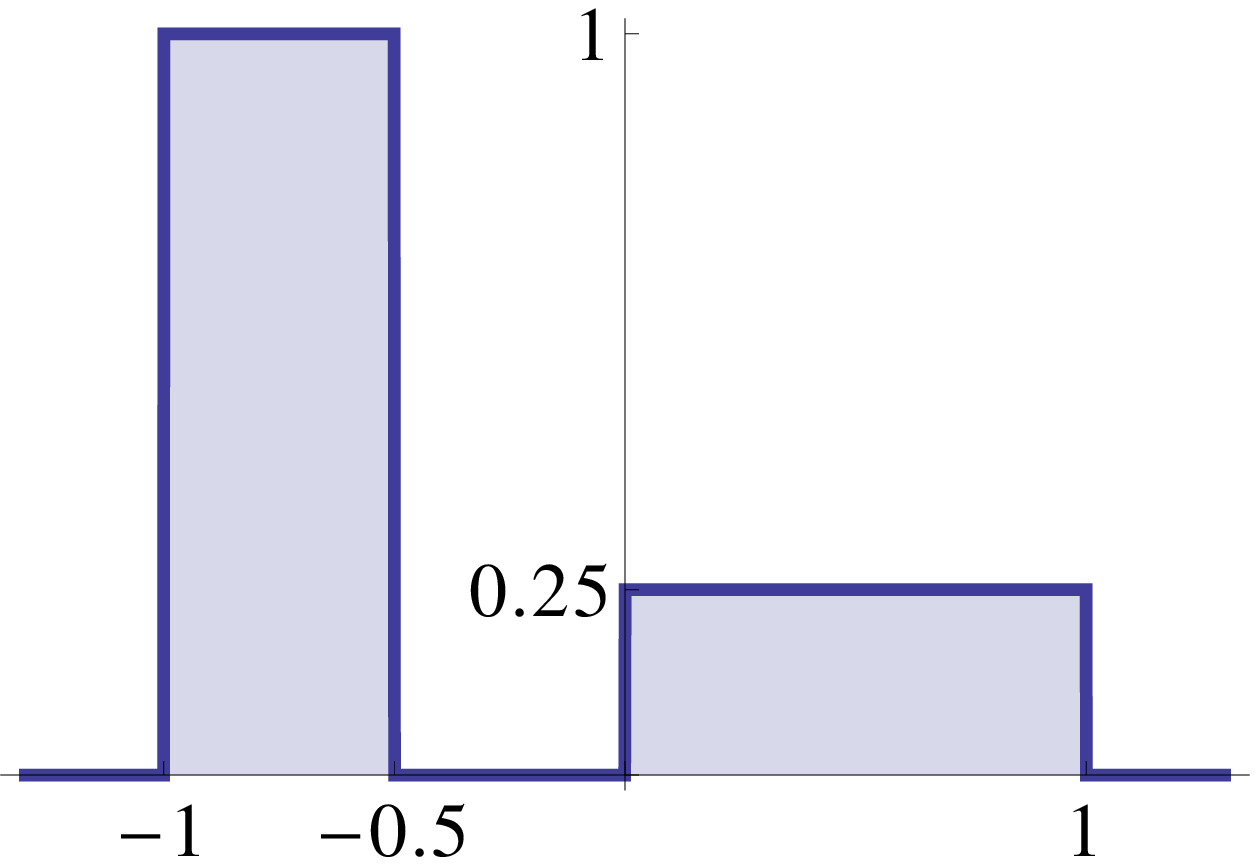} \\
a) $f(x)$ & & (b) $p(x)$
\end{tabular}
\protect\caption{}\label{fig:counterexample_non_stab}
\end{figure}
Thus $L=-1$, $R = 2$, $n=2$, $t_0=-1$, $t_1=-0.5$, $t_2=0$, $t_3=1$, $p_0=1$, $p_1=0$, $p_2=0.25$.
Then
\begin{align*}
X_2 &= L \mu[-1, t_2] + R \mu[t_2, 1] = -1 \cdot 0.5 + 2 \cdot 0.25 = 0, \\
X_1 &= L \mu[-1, t_1] + R \mu[t_1, 1] = -1 \cdot 0.5 + 2 \cdot 0.25 = 0.
\end{align*}
Therefore $X_2=X_1=0$ and the assumptions of Theorem~\ref{th:LR_averaging_stab} fail.
On the other hand, it follows from~\eqref{equ:derivative_LxRx} that for $x\in [-\alpha t_2, -\alpha t_1] = [0, 0.5\alpha]$ we have that
\[
\frac{1}{\alpha} \, f'_{\alpha}(x)= X_{2} + \frac{x+\alpha t_2}{t_3-t_2}(X_1-X_2) = 0.
\]
Hence $f_{\alpha}$ is constant on $[0, 0.5\alpha]$, and so it is not topologically equivalent to $f$, see Figure~\ref{fig:counterexample_non_stab_av}.
\begin{figure}[h]
\includegraphics[width=0.3\textwidth]{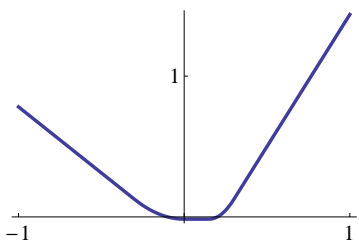}
\protect\caption{Graph of $f_{\alpha}$}\label{fig:counterexample_non_stab_av}
\end{figure}
Thus the assumptions of Theorem~\ref{th:LR_averaging_stab} are essential.
\end{example}

\def\cprime{$'$}
\providecommand{\bysame}{\leavevmode\hbox to3em{\hrulefill}\thinspace}
\providecommand{\MR}{\relax\ifhmode\unskip\space\fi MR }
\providecommand{\MRhref}[2]{%
  \href{http://www.ams.org/mathscinet-getitem?mr=#1}{#2}
}
\providecommand{\href}[2]{#2}

\end{document}